\documentclass[11pt, notitlepage]{article}
\usepackage{amssymb,amsmath,comment}
\usepackage{amsthm}
\catcode`\@=11 \@addtoreset{equation}{section}
\def\thesection{\arabic{section}}

\def\theequation{\thesection.\arabic{equation}}
\catcode`\@=12
\usepackage{colortbl}
\usepackage{hyperref}
\usepackage[mathscr]{eucal}
\usepackage{epsf}
\usepackage{esint}
\usepackage{a4wide}

\newcommand{\ra} {\rightarrow}

\newcommand{\De} {\Delta}
\newcommand{\la} {\lambda}

\newcommand{\noi} {\noindent}
\newcommand{\na} {\nabla}

\newcommand{\mb} {\mathbb}

\newcommand{\I}{\int\limits_}
 
\usepackage[all]{xy}
\catcode`\@=11
\def\theequation{\@arabic{\c@section}.\@arabic{\c@equation}}

\def\QED{\hfill {$\square$}\goodbreak \medskip}
\newtheorem{Theorem}{Theorem}[section]
\newtheorem{Lemma}[Theorem]{Lemma}
\newtheorem{Proposition}[Theorem]{Proposition}
\newtheorem{Corollary}[Theorem]{Corollary}
\newtheorem{Remark}[Theorem]{Remark}
\newtheorem{Definition}[Theorem]{Definition}

\def\XXint#1#2#3{{\setbox0=\hbox{$#1{#2#3}{\int}$ }
		\vcenter{\hbox{$#2#3$ }}\kern-.6\wd0}}

\usepackage[all]{xy}
\catcode`\@=11
\def\theequation{\@arabic{\c@section}.\@arabic{\c@equation}}
\catcode`\@=12

\begin{document}

{\vspace{0.01in}
		\title
{Pohozaev-type identities for classes of quasilinear elliptic local and nonlocal equations and systems, with applications}
		\author{ {\bf Gurdev Chand Anthal\footnote{School of Mathematical Sciences, Zhejiang Normal University, Jinhua 321004, People's Republic of China, e-mail: Gurdevanthal92@gmail.com}\; and Prashanta Garain\footnote{Indian Institute of Science Education and Research Berhampur, Berhampur, Ganjam, Odisha, India-760010
					email: pgarain92@gmail.com}  }}}
	\date{}

\maketitle

\begin{abstract}\noindent
In this article, we establish Pohozaev-type identities for a class of quasilinear elliptic equations and systems involving both local and nonlocal $p$-Laplace operators. Specifically, we obtain these identities in $\mb R^n$ for the purely anisotropic $p$-Laplace equations, the purely fractional $p$-Laplace equations, as well as for equations that incorporate both anisotropic and fractional $p$-Laplace features. We also extend these results to the corresponding systems. To the best of our knowledge, the identities we derive in the mixed case are new even when $p=2$. Finally, we illustrate some of the applications of our main results.
\end{abstract}

\maketitle

\noi {Keywords: Pohozaev identity, anisotropic $p$-Laplace equation, fractional $p$-Laplace equation, mixed local and nonlocal $p$-Laplace equation, system of quasilinear equations, non-existence.}

\noi{\textit{2020 Mathematics Subject Classification: 35R11, 35J92, 35A01, 35J62.}

\bigskip

\tableofcontents
\section{Introduction}
In this article, our primary objective is to derive a Pohozaev-type identity for the following class of quasilinear equations:
\begin{equation}\label{meqn}
\mathcal{M}_{\alpha,\beta}\,u:=-\alpha H_p u+\beta (-\Delta_p)^s u=f(u)\text{ in }\mathbb{R}^n,
\end{equation}
and for the following system of equations:
\begin{equation}\label{msys}
\begin{split}
-\alpha H_p u+\beta (-\Delta_p)^s u&=g_u(u,v)\text{ in }\mathbb{R}^n\\
-\alpha H_p u+\beta(-\Delta_p)^s u&=g_v(u,v)\text{ in }\mathbb{R}^n.
\end{split}
\end{equation}
Here, the parameters satisfy $0<s<1<p<\infty$ and $
(\alpha,\beta)\in S$, where
\begin{equation}\label{S}
S=\{(1,0), (1,\gamma), (0,1)\},
\end{equation} 
where $\gamma>0$. The nonlinearity $f:\mathbb{R}\to\mathbb{R}$ is a continuous function subject to the condition:
$$
(f_1)~~~~~~~~|f(t)| \leq C|t|^{p-1},~t \in \mb R,
$$
for some constant $C>0$. Here $g:\mathbb{R}^2\to\mathbb{R}$ is a continuously differentiable function subject to the condition
$$
(g_1)~~~~~~~~|g_t(t,s)| + |g_s(t,s)| \leq C (|t|^{p-1}+|s|^{p-1}),~(t,s) \in \mb R^2,
$$
for some constant $C>0$.

The operator $H_p$ denotes the anisotropic $p$-Laplace operator, defined as
$$
H_p u=\text{div}(H(\nabla u)^{p-1}\nabla_{\xi}H(\nabla u)),\quad 1<p<\infty,
$$
where $H:\mathbb{R}^n\to[0,\infty)$ is a Finsler-Minkowski norm. That is $H$ is a strictly convex function belonging to $C^1(\mathbb{R}^n\setminus\{0\})$, satisfying the following conditions:
\begin{enumerate}
\item[(H1)] $H(x)=0$ iff $x=0$.

    \item[(H2)] $H(tx)=|t|H(x)$ for every $x\in\mathbb{R}^n$, $t\in\mathbb{R}$.
    
    \item[(H3)] There exist constants $C_1,C_2>0$ such that $C_1|x|\leq H(x)\leq C_2|x|$ for every $x\in\mathbb{R}^n$.
\end{enumerate}
The symbol $\nabla_\xi H(\na u)$ represents the gradient of $H(\nabla u)$ with respect to the $\xi$ variable, and $\nabla u(x)$ denotes the gradient of $u$ with respect to the $x$ variable. 
We present some examples now to provide a little more insight into the anisotropic p-Laplace operator; refer to \cite{BFKzamp, Xiathesis, MV} and the references therein.\\
\noi \textbf{Examples:} Let $x=(x_1,x_2,\ldots,x_n)\in\mathbb{R}^n$,
\begin{enumerate}
\item[(i)] then, for $q>1$, we define 
\begin{equation}\label{ex11}
H_q(x):=\Big(\sum_{i=1}^{n}|x_i|^q\Big)^\frac{1}{q},
\end{equation}
\item[(ii)] for $\lambda,\mu>0$, we define
\begin{equation}\label{ex2}
H_{\lambda,\mu}(x):=\sqrt{\lambda\sqrt{\sum_{i=1}^{N}x_i^{4}}+\mu\sum_{i=1}^{n}x_i^{2}}\,.
\end{equation}
\end{enumerate}
It can be verified from \cite{MV} that the functions $H_q, H_{\lambda,\mu}:\mathbb{R}^n\to[0,\infty)$, defined in \eqref{ex11} and \eqref{ex2}, respectively, satisfy the properties of Finsler-Minkowski norms.
\begin{Remark}\label{exrmk1}
For $i=1,2$ if $\lambda_i,\mu_i$ are positive real numbers such that $\frac{\lambda_1}{\mu_1}\neq\frac{\lambda_2}{\mu_2}$, then $H_{\lambda_1,\mu_1}$ and $H_{\lambda_2,\mu_2}$ given by \eqref{ex2} define two non-isometric norms in $\mathbb{R}^n$, see \cite{MV}.
\end{Remark}

\begin{Remark}\label{exrmk2}
Moreover, for $H=H_q$ given by \eqref{ex11}, we have
\begin{equation*}\label{ex}
H_{p} u=\sum_{i=1}^{n}\frac{\partial}{\partial x_i}\bigg(\Big(\sum_{k=1}^{n}\Big|\frac{\partial u}{\partial x_k}\Big|^{q}\Big)^\frac{p-q}{q}\Big|\frac{\partial u}{\partial x_i}\Big|^{q-2}\frac{\partial u}{\partial x_i}\bigg).
\end{equation*}
Therefore, $H_p$ become the $p$-Laplace operator $\Delta_p$, when $q=2$ and pseudo $p$-Laplace operator $S_p$, when $q=p$ as given below: 
\begin{equation*}\label{plap}
H_{p} u=
\begin{cases}
\Delta_p u:=\text{div}(|\nabla u|^{p-2}\nabla u),\,\text{if }q=2,\,1<p<\infty,\\
S_p u:=\sum_{i=1}^{n}\frac{\partial}{\partial x_i}\Big(|u_i|^{p-2}u_i\Big),\,\text{if }q=p\in(1,\infty),
\end{cases}
\end{equation*}
where $u_i:=\frac{\partial u}{\partial x_i}$, for $i=1,2,\ldots,n$.
\end{Remark}
In particular, the examples above indicate that the operator $H_p$ generalizes the following form of anisotropic quasilinear operator:
$$
\sum_{i=1}^{n}\frac{\partial}{\partial x_i}\bigg(\Big(\sum_{k=1}^{N}\Big|\frac{\partial u}{\partial x_k}\Big|^{q}\Big)^\frac{p-q}{q}\Big|\frac{\partial u}{\partial x_i}\Big|^{q-2}\frac{\partial u}{\partial x_i}\bigg).
$$
Furthermore, for $1<p<\infty$ and $0<s<1$, the fractional $p$-Laplace operator $(-\Delta_p)^s u$ is defined by
 \begin{align*}
     (-\De)_p^su(x) =\text{P.V.} \I{\mb R^n} \frac{|u(x)-u(y)|^{p-2}(u(x)-u(y)}{|x-y|^{n+ps}}\,dxdy,\quad x\in\mathbb{R}^n
 \end{align*}
where P.V. stands for the Cauchy principal value. Therefore, the equation \eqref{meqn} and the system \eqref{msys} encompass a broad class of quasilinear problems, depending on the values of the parameters $\alpha$ and $\beta$. Specifically, we have:
$$
\mathcal{M}_{1,0}\,=-H_p,\quad \mathcal{M}_{1,\gamma}\,=-H_p+\gamma(-\Delta_p)^s\text{ and } \mathcal{M}_{0,1}\,=(-\Delta_p)^s,
$$
for $\gamma>0$.

The classical Pohozaev identity, established in the seminal work of Pohozaev \cite{Pohozaev}, asserts that any classical solution $u\in C^2(\Omega)\cap C^1(\overline{\Omega})$ of the boundary value problem
\begin{equation*}
-\Delta u=g(u)\text{ in }\Omega,\quad u=0\text{ on }\partial\Omega
\end{equation*}
satisfies the following identity
\begin{equation*}
(2-n)\int_{\Omega}u g(u)\,dx+2n\int_{\Omega}G(u)\,dx=\int_{\partial\Omega}|\nabla u|^2\,x\cdot \eta\,dx,
\end{equation*}
where $\Omega\subset\mathbb{R}^n$ is a bounded smooth domain, $g$ is a continuous function on $\mathbb{R}$ such that $G'=g$,\,$G(0)=0$ and $\eta$ is the outward unit normal vector to $\partial\Omega$. This identity was later extended in \cite{N-N-S} to the case where $g=g(x,u)$. Based on Pohozaev's identity, nonexistence results for semilinear Dirichlet problems in star-shaped domains were obtained in \cite{B-N, Pohozaev}.

In the quasilinear setting, similar identities were derived in \cite{P-S} for $C^2$ solutions, including extensions to systems, higher-order equations, and both bounded and unbounded domains. For $W_0^{1,p}(\Omega)\cap L^\infty(\Omega)$ solutions, we refer to \cite{G-V}, and for $C^1$ solutions, see \cite{M-A-Squa}, which generalizes the results in \cite{P-S}. Additionally, in \cite{LLjmaa}, a Pohozaev identity was established for $p$-Laplace equations with solutions in $W^{1,p}(\mb R^n)$.

Moreover, in \cite{Med}, a Pohozaev-type identity was obtained for the following system posed on a bounded domain $\Omega\subset\mb R^n$:
\begin{equation*}
\begin{split}
-\Delta u &= f(u, v) \quad \text{in } \Omega, \\
-\Delta v &= g(u, v) \quad \text{in } \Omega,
\end{split}
\end{equation*}
for suitable nonlinearities $f$ and $g$. Extensions of such identities to systems involving the $p$-Laplacian can be found in \cite{Korman} and the references therein.

More recently, the Pohozaev identity has also been investigated for the anisotropic $p$-Laplace equation. In particular, we refer to \cite{WXia}, where the identity is established for $C^2$ solutions in bounded domains. Additionally, in \cite{Sciani}, the authors derive the identity for $C^1$ solutions, both in bounded domains and in the whole space $\mb R^n$. Notably, the result in $\mb R^n$ is obtained by extending the identity derived in bounded domains.

In this article, one of our main objectives is to establish a Pohozaev-type identity for the anisotropic $p$-Laplace equation and its associated systems without assuming $C^1$ regularity. Specifically, we work within the framework of $W^{1,p}(\mb R^n)$ solutions.

In the purely nonlocal setting, Pohozaev-type identities for the fractional Laplace equation have been established in bounded domains in \cite{Amb21, Weth, Oton1, Oton2}, and in the whole space $\mb R^n$ in \cite{Chang}. For systems involving fractional Laplace equations, we refer to \cite{MFL}. In the case of the fractional $p$-Laplace equation on $\mb R^n$, a Pohozaev-type identity was recently derived in \cite{Ambrosio}.
For problems involving both local and nonlocal operators—the so-called mixed case—the only known results for Pohozaev-type identities in bounded domains appear in \cite{AGS, Q, biswas2024pohozaevidentitymixedlocalnonlocal}.

In this work, we aim to establish Pohozaev-type identities for three distinct cases: the purely anisotropic $p$-Laplace equation, the purely fractional $p$-Laplace equation, and the mixed case involving both anisotropic and fractional $p$-Laplacians. To the best of our knowledge, our result is new even in the mixed case when $p=2$. Furthermore, we extend these identities to the corresponding systems of equations. Our approach provides a direct proof by employing an appropriately chosen test function involving difference quotient of the solutions. To this end, we follow the methodology developed in \cite{Ambrosio, LLjmaa}.

The remainder of the paper is organized as follows. In Section 2, we introduce the functional framework, present the main results and mention the notations. Section 3 is devoted to the proofs of these results. Finally, in Section 4, we discuss some applications of our main theorems.
\section{Functional setting, main results and notations}
\subsection{Functional setting}
\textbf{Sobolev spaces:} {Let $\Omega\subset\mathbb{R}^n$ be an open and connected subset and let $0<s<1<p<\infty$. The Sobolev space $W^{1,p}(\Omega)$ is defined to be the space of functions $u:\Omega\to\mathbb{R}$ in $L^p(\Omega)$ such that the partial derivatives $\frac{\partial u}{\partial x_i}$ for $1\leq i\leq n$ exist in the weak sense and belong to $L^p(\Omega)$. The space $W^{1,p}(\Omega)$ is a Banach space (see \cite{LC}) equipped with the norm:
$$
\|u\|_{W^{1,p}(\Omega)} =\left(\int_{\Omega}|u|^p\,dx+\int_{\Omega}H(\nabla u)^p\,dx\right)^\frac{1}{p}
$$
where $\nabla u=\Big(\frac{\partial u}{\partial x_1},\ldots,\frac{\partial u}{\partial x_n}\Big)$.  
}The fractional Sobolev space $W^{s,p}(\Omega)$ for $0<s<1<p<\infty$, is defined by
$$
W^{s,p}(\Omega)=\Bigg\{{u:\Omega\to\mathbb{R}:\,}u\in L^p(\Omega),\,\frac{|u(x)-u(y)|}{|x-y|^{\frac{n}{p}+s}}\in L^p(\Omega\times \Omega)\Bigg\}
$$
under the norm
$$
\|u\|_{W^{s,p}(\Omega)}=\left(\int_{\Omega}|u|^p\,dx+\int_{\Omega}\int_{\Omega}\frac{|u(x)-u(y)|^p}{|x-y|^{n+ps}}\,dx\,dy\right)^\frac{1}{p}.
$$
We refer to \cite{Hitchhikersguide} and the references therein for more details on fractional Sobolev spaces. Throughout the rest of the article, for $1<p<\infty$ and $0<s<1$, we denote by
\begin{align*}
 	\|H(\nabla u)\|_{L^p(\mb R^n)}^p=\int_{\mathbb{R}^n}H(\nabla u)^p\,dx\quad \forall u\in W^{1,p}(\mathbb{R}^n)
 \end{align*}
 and 
 \begin{align*}
 	[u]_{s,p}^p =\I{\mb R^n}\I{\mb R^n}\frac{|u(x)-u(y)|^p}{|x-y|^{n+ps}}\,dxdy\quad \forall u\in W^{s,p}(\mathbb{R}^n).
 \end{align*}
The following results are taken from \cite[Proposition 2.1]{PF20}, \cite[Proposition 1.2]{Xiathesis} and \cite[Lemma 5.9]{Heinonen}.
\begin{Lemma}\label{Happ}
For every $x\in\mathbb{R}^N\setminus\{0\}$ and $t\in\mathbb{R}\setminus\{0\}$, we have
\begin{enumerate}
\item[(A)] $x\cdot\nabla H(x)=H(x)$. 
\item[(B)] $\nabla H(tx)=\text{sign}(t)\nabla H(x)$.
\item[(C)] $|\nabla H(x)|\leq C$, for some fixed positive constant $C$.
\end{enumerate}
\end{Lemma}
\noi \textbf{Solution space:} For a given $(\alpha,\beta)\in S$ where $S$ is defined in \eqref{S}, we define the space $X_{\alpha,\beta}$ by
\begin{equation}\label{solsp}
X_{\alpha,\beta}=
\begin{cases}
W^{1,p}(\mathbb{R}^n)\text{ if } (\alpha,\beta)=(1,0).\\
W^{1,p}(\mathbb{R}^n)\cap W^{s,p}(\mathbb{R}^n)\cap C^{0,l}(\mathbb{R}^n)\text{ for some } l>s, \text{ if } (\alpha,\beta)=(1,\gamma),\text{ where }\gamma>0.\\
\Big(W^{s,p}(\mathbb{R}^n)\cap C^{1,1}(\mathbb{R}^n)\Big)\cup \Big(W^{1,p}(\mathbb{R}^n)\cap W^{s,p}(\mathbb{R}^n)\Big)\text{ if } (\alpha,\beta)=(0,1).
\end{cases}
\end{equation}
\textbf{Weak solution:}
Before stating our main results, we define the notion of weak solutions for the equation \eqref{meqn} and system \eqref{msys}.
\begin{Definition}\label{wksolequ}
Let $0<s<1<p<\infty$. Suppose $(\alpha,\beta)\in S$, where $S$  is defined in \eqref{S} and $f:\mb R\to\mb R$ is a continuous function. Then we say that $u\in X_{\alpha,\beta}$ defined in \eqref{solsp} is a weak solution of the equation \eqref{meqn}, if for every $\phi\in C_c^{1}(\mathbb{R}^n)$, we have
\begin{equation}\label{wksoldef}
\begin{split}
&\alpha\int_{\mathbb{R}^n}H(\nabla u)^{p-1}\nabla_{\xi}H(\nabla u)\cdot \na \phi\,dx+\beta\int_{\mathbb{R}^n}\int_{\mathbb{R}^n}J_p(u(x)-u(y))(\phi(x)-\phi(y))\,d\mu_{s,p}=\int_{\mathbb{R}^n}f(u)\phi\,dx,
\end{split}
\end{equation}
where $d\mu_{s,p}$ and $J_p$ are defined in \eqref{mu} and \eqref{jp} respectively.
\end{Definition}

\begin{Definition}\label{wksolsys}
Let $0<s<1<p<\infty$. Suppose $(\alpha,\beta)\in S$, where $S$ is defined in \eqref{S} and $g: \mb R^2 \rightarrow \mb R$ is a continuously differentiable function. Then we say that $(u,v)\in X_{\alpha,\beta}\times X_{\alpha,\beta}$ (where $X_{\alpha,\beta}$ is defined in \eqref{solsp}) is a weak solution of the system of equations \eqref{msys}, if for every $\phi,\psi\in C_c^{1}(\mathbb{R}^n)$, the following two equations hold simultaneously:
\begin{equation}\label{wksysdef1}
\begin{split}
\alpha\int_{\mathbb{R}^n}&H(\nabla u)^{p-1}\nabla_{\xi}H(\nabla u)\cdot\na\phi\,dx+\beta\int_{\mathbb{R}^n}\int_{\mathbb{R}^n}J_p(u(x)-u(y))(\phi(x)-\phi(y))\,d\mu_{s,p}\\
&=\int_{\mathbb{R}^n}g_u(u,v)\phi\,dx,
\end{split}
\end{equation}
and
\begin{equation}\label{wksysdef2}
\begin{split}
\alpha\int_{\mathbb{R}^n}&H(\nabla v)^{p-1}\nabla_{\xi}H(\nabla v)\cdot\na\psi\,dx+\beta\int_{\mathbb{R}^n}\int_{\mathbb{R}^n}J_p(v(x)-v(y))(\psi(x)-\psi(y))\,d\mu_{s,p}\\
&=\int_{\mathbb{R}^n}g_v(u,v)\psi\,dx,
\end{split}
\end{equation}
where $d\mu_{s,p}$ and $J_p$ are defined in \eqref{mu} and \eqref{jp} respectively.
\end{Definition}
\begin{Remark}\label{weldef}
In light of Lemma \ref{Happ}, we note that Definitions \ref{wksolequ} and \ref{wksolsys} are well-defined.
\end{Remark}
\subsection{Main results}
Our main results of this article read as follows:
\begin{Theorem}\label{t1.1}
		Let $0<s<1<p<\infty$, $f:\mb R\to\mb R$ is a continuous function and satisfies the hypothesis $(f_1)$ and $F(t) = \I{0}^t f(s) ds$. Suppose that $(\alpha,\beta)\in S$ and $u \in X_{\alpha,\beta}$ is a weak solution of \eqref{meqn}. Then $u$ satisfies the following Pohozaev type identity:
		\begin{align*}
			\frac{\alpha(n-p)}{p} \|H(\na u)\|_{L^p(\mb R^n)}^p+ \frac{\beta(n-sp)}{p} [u]_{s,p}^p = n\I{\mb R^n} F(u)\,dx.
		\end{align*}
	\end{Theorem}

Next, we have the following result for the system of equations:
\begin{Theorem}\label{t1.2}
Let $0<s<1<p<\infty$. Suppose $g:\mb R^2\to\mb R$ is a continuous function and satisfies the hypothesis $(g_1)$. Assume that $(\alpha,\beta)\in S$ and $(u,v)\in X_{\alpha,\beta}\times X_{\alpha,\beta}$ be a weak solution of \eqref{msys} such that $g(u,v)\in L^1(\mb R^n)$. Then $(u,v)$ satisfies the following Pohozaev-type identity:
		\begin{align*}
			\frac{\alpha(n-p)}{p} \Big(\|H(\na u)\|_{L^p(\mb R^n)}^p+\|H(\na v)\|_{L^p(\mb R^n)}^p\Big)
 + \frac{\beta(n-sp)}{p} \Big([u]_{s,p}^p+[v]_{s,p}^p\big)=n\int_{\mathbb{R}^n} g(u,v)\,dx.
		\end{align*}
\end{Theorem}

  \begin{Corollary}\label{cor1}
    Suppose $f$ satisfies the general hypothesis
    $$
(f_2)~~~~~~~~|f(t)| \leq C(|t|^{p-1}+|t|^{q-1}),~\text{for}~t \in \mb R,
$$
and $g$ satisfies the general hypothesis
$$
(g_2)~~~~~~~~|g_t(t,s)| + |g_s(t,s)| \leq C (|t|^{p-1}+|s|^{p-1}),~(t,s) \in \mb R^2,
$$
for some constant $C>0$, where $p<q\leq p^*=\frac{np}{n-p}$ with $1<p<n$. Then $u\in L^\infty(\mb R^n)$ gives the implication $(f_2)$ to $(f_1)$ and $(g_2)$ to $(g_1)$. Therefore, Theorem \ref{t1.1} is valid for any $u\in X_{\alpha,\beta}\cap L^\infty(\mb R^n)$ when $f$ satisfies $(f_2)$ and Theorem \ref{t1.2} is valid for any $u\in X_{\alpha,\beta}\cap L^\infty(\mb R^n)$ when $g$ satisfies $(g_2)$ respectively. In particular, Theorem \ref{t1.1} extends \cite[Proposition 2.1]{LLjmaa} to the anisotropic case for bounded weak solutions of \eqref{meqn}.
    \end{Corollary}

\subsection{Notations}
The following notations will be used consistently throughout the paper unless stated otherwise.
 \begin{itemize}
 
\item For $0<s<1<p<\infty$, we denote by 
 \begin{equation}\label{mu}
 d\mu_{s,p}=\frac{dx dy}{|x-y|^{n+sp}},
 \end{equation}
 and
 \begin{equation}\label{jp}
 J_p(r)=|r|^{p-2}r\text{ for } r\in\mb R.
 \end{equation}
 \item For $t\in\mb R$, we denote by $t_\pm=\max\{\pm t,0\}$.
 
 \item If $v:\mb R^n\to \mb R$, we define the difference quotient of $v$ by $D_i v(x) = \frac{v(x+he_i)-v(x_i)}{h}$, where $e_i=(0,\ldots,1,\ldots,0)$ is the $i^\text{th}$ unit vector along the coordinate $x_i$, $i\in\{1,2,\ldots,n\}$. Further, if $v:\mb R^n\to\mb R^n$ is given by $v(x)=(v_1(x),\ldots,v_n(x))$, then we define $D_i v(x)=(D_i v_1(x),\ldots, D_i v_n(x))$, where $i\in\{1,2,\ldots,n\}$.
 
 \item The symbol $C$ will represent a positive constant. When $C$ depends on parameters $r_1,r_2,\ldots,r_k$, we denote this by $C=C(r_1,r_2,\ldots,r_k)$.

 \end{itemize}

\section{Proof of the main results}
\textbf{Proof of Theorem \ref{t1.1}:}
Let $\psi \in C_c^1(\mb R^n)$ be such that $0 \leq \psi \leq 1$ in $\mb R^n$, $\psi\equiv 1$ for $|x| \leq 1$ and $\psi(x)\equiv 0$ for $|x| \geq 2$. We define $ \psi_\la(x) =\psi(\la x)$ for all $x \in \mb R^n$ and $\la >0$. Note that for all $x \in \mb R^n$ and $\la >0$, we get
		\begin{align}\label{e2.2}
			0 \leq \psi_\la (x) \leq 1~\text{and}~|x||\na \psi_\la(x)| \leq C,
		\end{align}
		where the constant $C>0$ is independent of $\la$. Also, we define the $C^1$ vector field as 
        \begin{equation}\label{psil}
        \Psi_\la(x)=\psi_\la(x)x.
        \end{equation}
        Now using \eqref{e2.2}, we conclude that
			\begin{align*}
				|\Psi_\la(x)-\Psi_\la(y)| \leq C|x-y|,
			\end{align*}
			for some constant $C>0$ independent of $\la$ and so $\Psi_\la \in C_c^{0,1}(\mb R^n, \mb R^n)$. Let $(\alpha,\beta)\in S$. Then choosing the test function $\varphi=\psi_\la\sum_{j=1}^{n}x_j D_j u$ in \eqref{wksoldef}, we obtain
\begin{equation}\label{e4.8}
\begin{split}
   &\alpha\I{\mb R^n} H(\nabla u)^{p-1})\nabla_\xi H(\nabla u) \cdot \nabla \varphi dx+ \beta  \I{\mb R^n}\I{\mb R^n} {|u(x)-u(y)|^{p-2}(u(x)-u(y)) \varphi(x)-\varphi(y))}\,d\mu_{s,p}\\
   &\quad
   = \I{\mb R^n} f(u) \varphi dx.
   \end{split}
\end{equation}
We observe that 
            \begin{equation}\label{MeI}
            \begin{split}
         \alpha\I{\mb R^n}H(\nabla u)^{p-1}\nabla_\xi H(\nabla u) \cdot \nabla \varphi \,dx&= \alpha\I{\mb R^n}  H(\nabla u)^{p-1}\nabla_\xi H(\nabla u) \cdot \nabla \left( \sum\limits_{j=1}^n \psi_\la x_j D_ju\right)\,dx \\
        &:= I_1 + I_2 + I_3,
        \end{split}
            \end{equation}
            where
            $$
            I_1=\alpha\sum\limits_{i,j=1}^n \I{\mb R^n} H(\na u)^{p-1} \frac{\partial H(\nabla u)}{\partial \xi_i} \frac{\partial \psi_\la }{\partial x_i}x_j D_j u\,dx,
            $$
            $$
            I_2=\alpha\sum\limits_{i,j=1}^n\I{\mb R^n} H(\nabla u)^{p-1}\frac{\partial H(\nabla u)}{\partial \xi_i} \psi_\la D_j(u) \delta_{ij} \,dx,
            $$
            and
            $$
            I_3=\alpha\sum\limits_{i,j=1}^n \I{\mb R^n} H(\na u)^{p-1} \frac{\partial H(\nabla u)}{\partial \xi_i}  \frac{\partial (D_j u)}{\partial x_i} {\psi_\la} x_j\,dx:=I_{3,1}-I_{3,2},
            $$
            where
            $$
    I_{3,1}=\frac{\alpha}{p} \sum\limits_{j=1}^n \I{\mb R^n} D_j(H(\nabla u)^p) \psi_\lambda x_j \,dx,
    $$
    and
    $$
    I_{3,2}=\frac{\alpha}{p}\sum\limits_{j =1}^n \I{\mb R^n}(D_j(H(\nabla u)^p) -p H(\nabla u)^{p-1}\nabla_\xi H(\nabla u)\cdot D_j(\nabla u))\psi_\lambda x_j \,dx.
    $$
    Here $\delta_{ij}$ being the Kronecker symbol, that is, $\delta_{ij}=1$ when $i=j$ and $0$ otherwise. We observe that
\begin{equation}\label{nonl-II}
\begin{split}
&\beta\I{\mb R^n}\I{\mb R^n}{|u(x)-u(y)|^{p-2}(u(x)-u(y))(\varphi(x)-\varphi(y))}\,d\mu_{s,p}\\
=&\quad\beta\sum\limits_{j=1}^n \I{\mb R^n}\I{\mb R^n}{|u(x)-u(y)|^{p-2}(u(x)-u(y))(\psi_\lambda(x) x_j D_j u(x)-\psi_\lambda(y) y_jD_j u(y))}\,d\mu_{s,p}\\
:=&\quad J_{1} + J_{2}, 
\end{split}
\end{equation}
where
\begin{align*}
    J_{1} =\beta\sum\limits_{j=1}^n \I{\mb R^n}\I{\mb R^n}{|u(x)-u(y)|^{p-2}(u(x)-u(y))(\psi_\lambda(x) x_j -\psi_\lambda(y) y_j)D_j u(x)}\,d\mu_{s,p},
\end{align*}
and 
\begin{align*}
    J_{2} =&\beta\sum\limits_{j=1}^n \I{\mb R^n}\I{\mb R^n} {|u(x)-u(y)|^{p-2}(u(x)-u(y))(D_j u(x) - D_j u(y))\psi_\lambda(y)y_j}\,d\mu_{s,p}\\
    =& \beta\sum\limits_{j=1}^n \I{\mb R^n}\I{\mb R^n}{|u(x)-u(y)|^{p-2}(u(x)-u(y))(D_j u(x) - D_j u(y))\psi_\lambda(x)x_j}\,d\mu_{s,p}\\
    =& J_{2,1} - J_{2,2},
    \end{align*}
    where
    \begin{align*}
      J_{2,1} =\frac{\beta}{p} \sum\limits_{j=1}^n \I{\mb R^n}\I{\mb R^n}{D_j(|u(x)-u(y)|^p)\psi_\lambda(x) x_j}\,d\mu_{s,p},
    \end{align*}
    and
\begin{align*}
    J_{2,2}
    =\frac{\beta}{p}\sum\limits_{j=1}^n\I{\mb R^n}\I{\mb R^n}\Big(D_j(|u(x)-u(y)|^p)-p|u(x)-u(y)|^{p-2}(u(x)-u(y))(D_j u(x) - D_j u(y))\Big)\psi_\lambda(x)x_j\,d\mu_{s,p}.
\end{align*}
Therefore, taking into account \eqref{MeI} and \eqref{nonl-II}, equation \eqref{e4.8} reduces into
\begin{equation}\label{lnest-II}
I_1+I_2+I_{3,1}-I_{3,2}+(J_{1}+J_{2,1}-J_{2,2})=L,
\end{equation}
where
\begin{equation*}
L=\I{\mb R^n} f(u) \Big(\sum\limits_{j=1}^n \psi_\lambda x_j D_j u\Big)\, dx.
\end{equation*}
 \noi\textbf{Estimates of $I_1$ and $I_2$:} {We observe that
    \begin{equation}\label{estI1-I}
    \begin{split}
       \lim_{h\to 0} I_1&= {\alpha}\sum\limits_{i,j=1}^n \I{\mb R^n}H(\nabla u)^{p-1} \frac{\partial H(\nabla u)}{\partial \xi_i} x_j \frac{\partial u}{\partial x_j} \frac{\partial \psi_\la}{\partial x_i}\,dx ={\alpha} \I{\mb R^n} H(\nabla u)^{p-1} (\nabla_\xi H(\nabla u) \cdot \nabla \psi_\lambda)( x \cdot \nabla u)\,dx.
    \end{split}
    \end{equation}
    Further, using Lemma \ref{Happ}, we have
   \begin{equation}\label{estI2-I}
    \begin{split}
       \lim_{h\to 0} I_2 = {\alpha}\sum\limits_{i,j=1}^n \I{\mb R^n} H(\nabla u)^{p-1} \frac{\partial H(\nabla u)}{\partial \xi_i} {\psi_\la} \frac{\partial u}{\partial x_j} \delta_{ij}\,dx&= {\alpha}\I{\mb R^n} \psi_\lambda\,H(\nabla u)^{p-1} \nabla_\xi H(\nabla u) \cdot \nabla u\,dx\\
        &={\alpha} \I{\mb R^n} H(\nabla u)^p\,\psi_\lambda \, dx.
        \end{split}
    \end{equation}
\noi \textbf{Estimate of $I_{3,1}$:} We observe that for any {$K\in L^1(\mathbb{R}^n)$}, we have
    \begin{align}\label{e4.4}
       \I{\mb R^n} D_j K \,dx = \frac1h\I{\mb R^n} (K(x+he_j) -K(x))\,dx
       =\frac1h \left( \I{\mb R^n} K(x) \,dx - \I{\mb R^n} K(x) \,dx \right)
       =0.
    \end{align}
         Using \eqref{e4.4} and the property
         \begin{equation}\label{up}
         D_j(fg)=g\,D_j f+\delta_j f\,D_j g
         \end{equation}
         for every $j=1,2,\ldots,n$, where $\delta_j f(x)=f(x+he_j)$, , we have
    \begin{equation}\label{I31est-I}
    \begin{split}
       \lim_{h\to 0} I_{3,1}=& \lim_{h\to 0}\left\{\frac{{\alpha}}{p} \sum\limits_{j=1}^n \I{\mb R^n}D_j(H(\nabla u)^p \psi_\lambda x_j)-\frac{{\alpha}}{p}\sum\limits_{j=1}^n\I{\mb R^n}H(\nabla u(x+he_j))^pD_j(\psi_\lambda x_j) \,dx\right\}\\
        =&-\frac{{\alpha}}{p}\lim_{h\to 0}\sum\limits_{j=1}^n\I{\mb R^n}H(\nabla u(x+he_j))^pD_j(\psi_\lambda x_j) \,dx\\
        =&-\frac{{\alpha}}{p}\lim_{h\to 0}\sum\limits_{j=1}^n\I{\mb R^n}H(\nabla u(x+he_j))^pD_j(\psi_\lambda x_j) \,dx\\
        = & -\frac{{\alpha}}{p} \sum\limits_{j=1}^n \I{\mb R^n} H(\nabla u)^p \frac{\partial (\psi_\lambda x_j)}{\partial x_j}\,dx\\
        =&-\frac{n{\alpha}}{p} \I{\mb R^n} H(\nabla u)^p \psi_\lambda \,dx -\frac{{\alpha}}{p} \I{\mb R^n} H(\nabla u)^p \nabla \psi_\lambda \cdot x \,dx.
    \end{split}
    \end{equation}
\textbf{Estimate of $I_{3,2}+J_{2,2}$:} Using the convexity of the map $t \mapsto |t|^p$, $p>1$, we deduce that
    \begin{align}\label{e4.3}
    D_j(H(\nabla u)^p)& -p H(\nabla u)^{p-1}\nabla_\xi H(\nabla u)\cdot D_j(\nabla u)
   \geq 0,
    \end{align}
    and 
\begin{align}\label{e4.9}
    D_j(|u(x)-u(y)|)^p-p|u(x)-u(y)|^{p-2}(u(x)-u(y))(D_j u(x)-D_j u(y) \geq 0,
\end{align}
for every $x,y\in\mb R^n$.
Using \eqref{e4.4}, we get
    \begin{align}\label{e4.5}
        \I{\mb R^n} D_j( H(\nabla u)^p) \,dx = \I{\mb R^n} D_j(f(u)u) \,dx = \I{\mb R^n} D_j F(u) \,dx =0.    \end{align}
Further, using similar arguments as in \eqref{e4.4} for double integral, we have
\begin{align}\label{e4.10}
    \I{\mb R^n}\I{\mb R^n} {D_j(|u(x)-u(y)|^p)}\,d\mu_{s,p} =0.
\end{align}
Choosing $\phi=\sum_{j=1}^{n}D_j u$ as a test function in \eqref{wksoldef} and combining the estimates \eqref{e4.3}, \eqref{e4.9}, \eqref{e4.5} and \eqref{e4.10}, we have
\begin{align*}
    \Big|&I_{3,2} + J_{2,2}\Big|\\
    \leq& C \sum\limits_{j=1}^n\left( {\alpha}\I{\mb R^n}{\Big(}D_j(H(\nabla u)^p -p H(\nabla u)^{p-1} \nabla_\xi H(\nabla u) \cdot \nabla( D_ju ){\Big)} dx \right.\\
    &+ \left. \beta\I{\mb R^n}\I{\mb R^n}{\Big(}{D_j(|u(x)-u(y)|^p)-p|u(x)-u(y)|^{p-2}(u(x)-u(y))(D_j u(x) - D_j u(y))}{\Big)}\,d\mu_{s,p}\right)\\
    =&Cp \sum\limits_{j=1}^n \I{\mb R^n} (D_j F(u) -f(u) D_j u)dx\\
 =& Cp  \sum\limits_{j=1}^n\I{\mb R^n}(f(\tilde{u})-f(u))D_j u \,dx,
    \end{align*}
    where $C = C(\lambda)>0$ is a constant}, $\tilde{u}$ is between $u$ and $u(\cdot + he_j)$. Now using the H\"{o}lder's inequality in the above estimate, we obtain
    \begin{equation}\label{I32--I}
        \Big|I_{3,2}+J_{2,2}\Big| \leq C \sum\limits_{j =1}^n \| f(\tilde{u})-f(u)\|_{L^{p'}(\mathbb{R}^n)}\|D_j u\|_{L^p(\mathbb{R}^n)}\leq  C \sum\limits_{j =1}^n \| f(\tilde{u})-f(u)\|_{L^{p'}(\mathbb{R}^n)} \| \nabla u\|_{L^p(\mathbb{R}^n)},
    \end{equation}
    where $1/p + 1/p'=1$. Since $u( \cdot + he_j) \rightarrow u$ in $L^p(\mb R^n)$, we obtain $\tilde{u} \rightarrow u$ in $L^p(\mb R^n)$. Finally using $(f_1)$, we have $|f(\tilde{u})| \leq |\tilde{u}|^{p-1}$ and since $|\tilde{u}|^{p-1} \ra |u|^{p-1}$ in $L^{p'}(\mathbb{R}^n)$, by the generalized Lebesgue's dominated convergence theorem, we have
    \begin{equation}\label{fu}
        \lim_{h\to 0}\|f(\tilde{u})-f(u)\|_{{L^{p'}(\mathbb{R}^n)}}=0.
    \end{equation}
Using \eqref{fu} and that $\|\nabla u\|_{L^p(\mathbb{R}^n)} < \infty$ in \eqref{I32--I}, we obtain 
\begin{equation}\label{I32J22-II}
\lim_{h\to 0}(I_{3,2}+J_{2,2})=0.
\end{equation}
\textbf{Estimate of $J_1$:} Applying the divergence theorem in the domain $\mb R^n \setminus \overline{B_\mu(y)}$, with $y \in \mb R^n$ and $\mu >0$ fixed, we obtain
\begin{equation}\label{estJ1-II}
\begin{split}
   \lim_{h\to 0} J_1 =& \beta\sum\limits_{j=1}^n\I{\mb R^n}\I{\mb R^n} {|u(x)-u(y)|^{p-2}(u(x)-u(y))(\psi_\lambda(x)x_j-\psi_\lambda(y) y_j)u_{x_j}(x)}\,d\mu_{s,p}\\
   =&\beta\sum\limits_{j=1}^n\lim\limits_{\mu \rightarrow 0^+}\I{\mb R^n}\I{\mb R^n\setminus \overline {B_\mu(y)}}{|u(x)-u(y)|^{p-2}(u(x)-u(y))(\psi_\lambda(x)x_j-\psi_\lambda(y) y_j)u_{x_j}(x)}\,d\mu_{s,p}\\
=&\beta\sum\limits_{j=1}^n\lim\limits_{\mu \rightarrow 0^+}\I{\mb R^n}\I{\mb R^n\setminus \overline {B_\mu(y)}} \frac{\frac{\partial }{\partial x_j} (\frac1p|u(x)-u(y)|^p) (\psi_\lambda(x)x_j-\psi_\lambda(y) y_j)}{|x-y|^{n+ps}}\,dxdy\\
=& -\frac{\beta}{p}\sum\limits_{j=1}^n\lim\limits_{\mu \rightarrow 0^+}\I{\mb R^n}\I{\mb R^n\setminus \overline {B_\mu(y)}} \frac{|u(x)-u(y)|^p \frac{\partial}{\partial x_j} \psi_\lambda (x)x_j}{|x-y|^{n+sp}}\,dxdy\\
&+\frac{(n+sp)\beta}{p}\sum\limits_{j=1}^n\lim\limits_{\mu \rightarrow 0^+}\I{\mb R^n}\I{\mb R^n\setminus \overline {B_\mu(y)}} |u(x)-u(y)|^p \frac{(x_j-y_j)(\psi_\lambda(x) x_j-\psi_\lambda(y)y_j)}{|x-y|^{n+sp+2}}\,dxdy\\
&+\frac{\beta}{p}\sum\limits_{j=1}^n
\lim\limits_{ \mu\rightarrow 0^+} \I{\mb R^n} \I{\partial B_\mu(y)} |u(x)-u(y)|^p \frac {(y_j-x_j) (\psi_\lambda(x) x_j-\psi_\lambda(y) y_j)}{|x-y|^{n+sp+1}}\,d\sigma(y) dx\\
=&-\frac{\beta}{p} \lim_{\mu \rightarrow 0} \I{\mb R^n} \I{\mb R^n \setminus \overline{B_\mu(y)}}\frac{|u(x)-u(y)|^p}{|x-y|^{n+ps}}\text{div}(\Psi_\lambda(x) )\,dxdy\\
&+\frac{(n+sp)\beta}{p}\lim\limits_{\mu \rightarrow 0^+}\I{\mb R^n}\I{\mb R^n\setminus \overline {B_\mu(y)}} |u(x)-u(y)|^p \frac{(x-y)\cdot (\Psi_\lambda(x) -\Psi_\lambda(y))}{|x-y|^{n+sp+2}}\,dxdy\\
&+ \frac{\beta}{p} \lim\limits_{ \mu\rightarrow 0^+} \mu^{-n-1-sp} \I{\mb R^n} \I{\partial B_\mu(y)} |u(x)-u(y)|^p (y-x) \cdot (\Psi_\lambda(x) -\Psi_\lambda(y) )\,d\sigma(y) dx\\
=&-\frac{\beta}{p}  \I{\mb R^n} \I{\mb R^n }\frac{|u(x)-u(y)|^p}{|x-y|^{n+ps}}\text{div}\Psi_\lambda(x)\,dxdy\\
&+\frac{(n+sp)\beta}{p}\I{\mb R^n}\I{\mb R^n} |u(x)-u(y)|^p \frac{(x-y)\cdot (\Psi_\lambda(x) -\Psi_\lambda(y))}{|x-y|^{n+sp+2}}\,dxdy+ I\\
=&-\frac{\beta}{p}  \I{\mb R^n} \I{\mb R^n }\frac{|u(x)-u(y)|^p}{|x-y|^{n+ps}}\text{div}\Psi_\lambda(x)\,dxdy\\
&+\frac{(n+sp)\beta}{p}\I{\mb R^n}\I{\mb R^n} |u(x)-u(y)|^p \frac{(x-y)\cdot (\Psi_\lambda(x) -\Psi_\lambda(y))}{|x-y|^{n+sp+2}}\,dxdy,
\end{split}
\end{equation}
where we used that $I=0$. Indeed, to estimate $I$, we observe that, since $\Psi_\la$ has compact support, there exists  $R>0$ such that $\Psi_\la(x)-\Psi_\la(y)=0$, whenever ${(x,y)}\in S_{\mu}^c$, $\mu \in (0,1)$, where
$$
S_{\mu}=\{(x,y)\in B_R(0)\times B_R(0):|x-y|=\mu\}.
$$
Taking this into account along with the fact that $u\in C^{0,l}(\mathbb{R}^n)$ for some $l>s$, $\Psi_\la\in C^{0,1}(\mathbb{R}^n)$ and following the arguments of the proof of the estimate (3.3) in \cite[Page 2010]{Ambrosio}, we have $I=0$.  \\
\textbf{Estimate of $J_{2,1}$:} Using the property \eqref{up}, we have
\begin{equation*}
\begin{split}
    J_{2,1}=&\frac{\beta}{p} \sum\limits_{j=1}^n \I{\mb R^n}\I{\mb R^n}{D_j(|u(x)-u(y)|^p)\psi_\lambda(x) x_j}\,d\mu_{s,p}\\
    =& \frac{\beta}{p} \sum\limits_{j=1}^n\I{\mb R^n}\I{\mb R^n} {D_j( |u(x)-u(y)|^p \psi_\la(x)x_j)}\,d\mu_{s,p}\\
    &- \frac{\beta}{p} \sum\limits_{j=1}^n \I{\mb R^n}\I{\mb R^n} {|u(x+he_j)-u(y+he_j)|^p D_j(\psi_\la(x)x_j)}\,d\mu_{s,p}\\
    =& -\frac{\beta}{p} \sum\limits_{j=1}^n\I{\mb R^n}\I{\mb R^n} {|u(x+he_j)-u(y+he_j)|^p D_j(\psi_\la(x)x_j)}\,d\mu_{s,p}.
    \end{split}
    \end{equation*}
    Therefore,
    \begin{equation}\label{estJ21-II}
\begin{split}
    \lim_{h\to 0}J_{2,1} &= -\frac{\beta}{p} \sum\limits_{j=1}^n \I{\mb R^n}\I{\mb R^n}{|u(x)-u(y)|^p \frac{\partial(\psi_\la(x)x_j)}{\partial x_j}}\,d\mu_{s,p} \\
    &=-\frac{n\beta}{p} \I{\mb R^n}\I{\mb R^n}{|u(x)-u(y)|^p\psi_\lambda(x)}\,d\mu_{s,p}-\frac{\beta}{p}\I{\mb R^n}\I{\mb R^n}{|u(x)-u(y)|^p \nabla \psi_\lambda(x) \cdot x}\,d\mu_{s,p}.
    \end{split}
    \end{equation}
    \textbf{Estimate of $L$:} Using integration by parts, we have
    \begin{equation}\label{RHS-I}
    \begin{split}
       \lim_{h\to 0}L&=\lim_{h\to 0}\I{\mb R^n} f(u) \Big(\sum\limits_{j=1}^n \psi_\lambda(x) x_j D_j u\Big)\, dx\\
        & =  \I{\mb R^n} f(u) \psi_\lambda(x) \nabla u\cdot x \,dx=-\left(n \I{\mb R^n} \psi_\la(x)F(u(x))\,dx +\la \I{\mb R^n}F(u(x))x\cdot \na \psi_\la(x)\,dx\right).
        \end{split}
    \end{equation}
Combining the estimates \eqref{estI1-I}, \eqref{estI2-I}, \eqref{I31est-I}, \eqref{I32J22-II}, \eqref{estJ1-II}, \eqref{estJ21-II} and \eqref{RHS-I} in \eqref{lnest-II}, we obtain
\begin{align}\nonumber \label{e4.11}
  -&\left(n \I{\mb R^n} \psi_\la(x)F(u(x))\,dx +\la \I{\mb R^n}F(u(x))x\cdot \na \psi_\la(x)\,dx\right)\\ \nonumber
  =  & {\alpha}\I{\mb R^n} H(\nabla u)^{p-1} (\nabla_\xi H(\nabla u) \cdot \nabla \psi_\lambda(x))( x \cdot \nabla u)\,dx+ \frac{(p-n){\alpha}}{p} \I{\mb R^n} H(\nabla u)^p \psi_\lambda(x) \,dx \nonumber
   -\frac{{\alpha}}{p} \I{\mb R^n} H(\nabla u)^p \nabla \psi_\lambda(x) \cdot x \,dx\\ \nonumber
   &-\frac{n\beta}{p} \I{\mb R^n}\I{\mb R^n}{|u(x)-u(y)|^p\psi_\lambda(x)}\,d\mu_{s,p}-\frac{\beta}{p}\I{\mb R^n}\I{\mb R^n}{|u(x)-u(y)|^p \nabla \psi_\lambda(x) \cdot x}\,d\mu_{s,p}\\ \nonumber
   &-\frac{\beta}{p}  \I{\mb R^n} \I{\mb R^n }{|u(x)-u(y)|^p}\text{div}\Psi_\lambda(x)\,d\mu_{s,p}\\
&+\frac{(n+sp)\beta}{p}\I{\mb R^n}\I{\mb R^n} |u(x)-u(y)|^p \frac{(x-y)\cdot (\Psi_\lambda(x) -\Psi_\lambda(y))}{|x-y|^{n+sp+2}}\,dxdy.
\end{align}
           Finally, we pass on to the limits $\lambda \rightarrow 0$ in \eqref{e4.11}. To this end, using \eqref{e2.2}, we notice that, for every $x,y \in \mb R^n$ with $x \ne y$ and $\la >0$,
	\begin{align}\label{e2.7}
		\left| \text{div}(\Psi_\la(x))+\text{div}(\Psi_\la(y))-(n+qs)\frac{(\Psi_\la(x)-\Psi_\la(y))\cdot(x-y)}{|x-y|^2}\right| \leq C,
	\end{align}
	for some constant $C >0$ independent of $\la$.
	Using \eqref{e2.2}, \eqref{e2.7} along with the pointwise convergences $\psi_\la(x) \ra 1$, $\na \psi_\la(x)\cdot x \ra 0$, $\Psi_\la (x) \ra x$ and $\text{div}\, \Psi_\la(x) \ra n$ for all $x \in \mb R^n$ as $\la \ra 0$ and the facts that $u \in X_{\alpha,\beta},\,F(u)\in L^1(\mb R^n)$, we apply the Lebesgue's dominated convergence theorem to obtain
	\begin{align*}
	    \frac{\alpha(n-p)}{p} \| H(\nabla u)\|_{L^p(\mb R^n)}^p + \frac{\beta(n-sp)}{p} [u]_{s,p}^p = n\I{\mb R^n} F(u)\, dx,
	\end{align*}
    which is the required identity. \QED
    \textbf{Proof of Theorem \ref{t1.2}.} Recall $\Psi_\la(x)$ as defined in \eqref{psil}. Then, for any $(\alpha,\beta)\in S$, testing the equation \eqref{wksysdef1} with $\psi_\la\sum_{j=1}^{n}x_j D_j u$ and the equation \eqref{wksysdef2} with $\psi_\la\sum_{j=1}^{n}x_j D_j v$ and then adding the resulting equations, we get
    \begin{align*}
     I_1+I_2+I_{3,1}-I_{3,2}+(J_{1}+J_{2,1}-J_{2,2})+I_1^\prime+I_2^\prime+I_{3,1}^\prime-I_{3,2}^\prime+(J_{1}^\prime+J_{2,1}^\prime-J_{2,2}^\prime)=L^\prime,
    \end{align*}
    where the integrals involving $I$ and $J$ are the same as in the proof of Theorem \ref{t1.1}, the integrals involving $I^\prime$ and $J^\prime$ are the corresponding integrals with $v$ in place of $u$ and
    $$
L^\prime=\I{\mb R^n} \left((g_u(u,v) \Big(\sum\limits_{j=1}^n \psi_\lambda(x) x_j D_j u\Big) + g_v(u,v) \Big(\sum\limits_{j=1}^n \psi_\lambda(x) x_j D_j v\Big)\right)\, dx.
$$
Now the limit as $h \rightarrow 0$ of $I_1,\,I_2,\,I_{3,1}, J_1, J_2, J_{2,1},\,I_1^\prime,\,I_2^\prime,\,I_{3,1}^\prime,\,J_{1}^\prime,\text{ and }J_{2,1}^\prime$ can be evaluated on the similar lines of proof of Theorem \ref{t1.1}. Finally, we will estimate $I_{3,2}+ I_{3,2}^\prime+J_{2,2}+J_{2,2}^\prime$. For this, first we note that similar to the equations \eqref{e4.3}, \eqref{e4.9}, \eqref{e4.5} and \eqref{e4.10}, we also have
    \begin{align}\label{e4.23}
    D_j(H(\nabla v)^p)& -p H(\nabla v)^{p-1}\nabla_\xi H(\nabla v)\cdot D_j(\nabla v)
   \geq 0,
    \end{align}
\begin{align}\label{e4.24}
    D_j(|v(x)-v(y)|)^p-p|v(x)-v(y)|^{p-2}(v(x)-v(y))(D_j v(x)-D_j v(y) \geq 0,\text{ for every } x,y\in\mb R^n,
\end{align}
    \begin{align}\label{e4.25}
        \I{\mb R^n} D_j( H(\nabla v)^p) \,dx  = \I{\mb R^{n}} D_j g(u,v) \,dx =0,
        \end{align}
        and
\begin{align}\label{e4.26}
    \I{\mb R^n}\I{\mb R^n} {D_j(|v(x)-v(y)|^p)}\,d\mu_{s,p} =0.
\end{align}
Now choosing $\phi= \sum_{j=1}^n D_j u$ and $\psi = \sum_{j=1}^n D_jv$ as a test function in \eqref{wksysdef1} and \eqref{wksysdef2} respectively and combining the estimates \eqref{e4.3}, \eqref{e4.9}, \eqref{e4.5}, \eqref{e4.10}, \eqref{e4.23}, \eqref{e4.24}, \eqref{e4.25}, and \eqref{e4.26}, we have
\begin{equation*}
    \begin{split}
        \Big|&I_{3,2}+ J_{2,2}+ I_{3,2}^\prime + J_{2,2}^\prime\Big|\\
    \leq& C \sum\limits_{j=1}^n\left({\alpha} \I{\mb R^n}{\Big(}D_j(H(\nabla u)^p -p H(\nabla u)^{p-1} \nabla_\xi H(\nabla u) \cdot \nabla( D_ju ){\Big)}dx \right.\\
    &+ \beta\I{\mb R^n}\I{\mb R^n}{\Big(}{D_j(|u(x)-u(y)|^p)-p|u(x)-u(y)|^{p-2}(u(x)-u(y))(D_j u(x) - D_j u(y))}{\Big)}\,d\mu_{s,p}\\
    &+{\alpha} \I{\mb R^n}{\Big(}(D_j(H(\nabla v)^p -p H(\nabla v)^{p-1} \nabla_\xi H(\nabla v) \cdot \nabla( D_jv ){\Big)} dx \\
    &+\left. \beta\I{\mb R^n}\I{\mb R^n}{\Big(}{D_j(|v(x)-v(y)|^p)-p|v(x)-v(y)|^{p-2}(v(x)-v(y))(D_j v(x) - D_j v(y))}{\Big)}\,d\mu_{s,p}\right)\\
    =&Cp \sum\limits_{j=1}^n \I{\mb R^n}{\Big (}D_j g(u,v) -g_u(u,v) D_j u-g_v(u,v)D_jv{\Big)}\,dx\\
 =& Cp  \sum\limits_{j=1}^n\I{\mb R^n}((g_u(\tilde{u},\tilde{v})-g_u(u,v))D_j u + (g_v(\tilde{u},\tilde{v})-g_v(u,v))D_jv)\,dx,
    \end{split}
\end{equation*}
where again $C = C(\lambda)>0$ is a constant, $(\tilde{u},\tilde{v})$ lies between $(u,v)$ and $(u(\cdot + h e_j), v(\cdot+he_j))$. In view of H\"older's inequality, we have
\begin{equation}\label{e4.27}
    \begin{split}
       \Big|I_{3,2}&+ J_{2,2}+ I_{3,2}^\prime + J_{2,2}^\prime\Big|\\  \leq& C\sum\limits_{j=1}^n (\| g_u(\tilde{u},\tilde{v})-g_u(u,v)\|_{L^{p^\prime}(\mb R^n)} \|D_j u\|_{L^p(\mb R^n)} + \| g_v(\tilde{u},\tilde{v})-g_v(u,v)\|_{L^{p^\prime}(\mb R^n)} \|D_j v\|_{L^p(\mb R^n)})\\
       \leq& C\sum\limits_{j=1}^n (\| g_u(\tilde{u},\tilde{v})-g_u(u,v)\|_{L^{p^\prime}(\mb R^n)} \|\nabla u\|_{L^p(\mb R^n)} + \| g_v(\tilde{u},\tilde{v})-g_v(u,v)\|_{L^{p^\prime}(\mb R^n)} \|\nabla v\|_{L^p(\mb R^n)}),
    \end{split}
\end{equation}
 where $1/p + 1/p'=1$. Since $u( \cdot + he_j) \rightarrow u$ and $v( \cdot + he_j) \rightarrow v$  in $L^p(\mb R^n)$, we obtain $\tilde{u} \rightarrow u$ and $\tilde{v} \rightarrow v$ in $L^p(\mb R^n)$. Finally using $(g_1)$, we have $|g_u(\tilde{u}, \tilde{v})| \leq |\tilde{u}|^{p-1}+|\tilde{v}|^{p-1}$ and $|g_v(\tilde{u}, \tilde{v})| \leq |\tilde{u}|^{p-1}+|\tilde{v}|^{p-1}$ and since $|\tilde{u}|^{p-1} + |\tilde{v}|^{p-1} \ra |u|^{p-1} + |\tilde{v}|^{p-1}$ in $L^{p'}(\mathbb{R}^n)$, by the generalized Lebesgue's dominated convergence theorem, we have
    \begin{equation}\label{e4.28}
        \lim_{h\to 0}\|g_u(\tilde{u},\tilde{v})-g_u(u,v)\|_{{L^{p'}(\mathbb{R}^n)}}=0,
    \end{equation}
    and
    \begin{equation}\label{e4.29}
        \lim_{h\to 0}\|g_v(\tilde{u},\tilde{v})-g_v(u,v)\|_{{L^{p'}(\mathbb{R}^n)}}=0,
    \end{equation}
Using equations \eqref{e4.28}, \eqref{e4.29}, $\|\nabla u\|_{L^p(\mathbb{R}^n)} < \infty$, and $\|\nabla v\|_{L^p(\mathbb{R}^n)} < \infty$ in \eqref{e4.27}, we obtain 
\begin{equation*}
\lim_{h\to 0}(I_{3,2}+ J_{2,2}+ I_{3,2}^\prime + J_{2,2}^\prime)=0.
\end{equation*}
The rest of the proof follows by proceeding along the lines of the proof of Theorem \ref{t1.1} and taking into account the following identity:
	\begin{align*}
		\I{\mb R^n}\Big((\Psi_\la(x)\cdot \na u) g_u(u,v) &+ (\Psi_\la(x) \cdot \na v)g_v(u,v)\Big)\,dx = \I{\mb R^n} \Psi_\la(x) \cdot \na g(u,v)\,dx \\=& -\left(\I{\mb R^n} n\,\psi_\la(x) g(u,v)\,dx + \I{\mb R^n} g(u,v) x \cdot \na \psi_\la(x)\,dx \right).
	\end{align*}
	This completes the proof. \QED

\section{Applications}
In this section, we present some applications of the Pohozaev identities obtained in the previous section.
\begin{Proposition}\label{PNE}
    Let $0<s<1<p<n$ and $q>1$. Then the problem 
    \begin{align*}
   -\alpha H_p u + \beta (-\De_p)^s u=\lambda t_+^{q-1}-\mu t_-^{q-1}~\text{in}~\mb R^n,
\end{align*}
where $\la,\mu\in\mb R$ admits no non-trivial weak solution $u \in X_{\alpha,\beta} \cap L^\infty(\mb R^n) \cap L^q(\mb R^n)$ in the following cases
\begin{enumerate}
    \item $(\alpha,\beta)=(1,0)$ and $  q \ne p^\ast=\frac{np}{n-p} $,
    \item $(\alpha,\beta)=(0,1)$ and $  q \ne p_s^\ast=\frac{np}{n-sp} $,
    \item $(\alpha,\beta)=(1,1)$ and either $q \geq p^\ast$ or $q \leq p_s^\ast$.
\end{enumerate}
\end{Proposition}
 \begin{proof} 
    Let $f(t) = \lambda t_+^{q-1}-\mu t_-^{q-1}$ and $F(t) = \I{0}^t f(s) ds$. Suppose $u \in X_{\alpha,\beta} \cap L^\infty(\mb R^n) \cap L^q(\mb R^n)$ is a weak solution of \eqref{meqn}. Then $f$ satisfies $(f_1)$. Thus, by testing the equation \eqref{wksoldef} with $u$ itself and using Lemma \ref{Happ}, we get 
        \begin{equation}\label{a2eq1}
            \alpha\| H(\nabla u)\|_{L^p(\mb R^n)}^p + \beta [u]_{s,p}^p= \I{\mb R^n}uf(u) \,dx.
        \end{equation}
        Also, since $u  \in L^\infty(\mb R^n)$, we have $|f(u) | \leq C |u|^{p-1}$ for some $C>0$ depending on $\|u\|_{L^\infty(\mb R^n)}$. By observing $F(t) =\frac{tf(t)}{q}$, we obtain by Theorem \ref{t1.1} that 
        \begin{equation}\label{a2eq2}
            \frac{\alpha(n-p)}{np}\| H(\nabla u)\|_{L^p(\mb R^n)}^p + \frac{\beta(n-ps)}{np}[u]_{s,p}^p = \frac1q \I{\mb R^n}uf(u)dx.
       \end{equation}
       Now, let us consider the three cases separately.\\
   \textbf{Case 1.} $(\alpha,\beta)=(1,0)$ and $  q \ne p^\ast=\frac{np}{n-p} $.\\
       In this case, on comparing the equations \eqref{a2eq1} and \eqref{a2eq2}, we get
        \begin{align*}
            \frac{n-p}{np}\| H(\nabla u)\|_{L^p(\mb R^n)}^p = \frac1q \| H(\nabla u)\|_{L^p(\mb R^n)}^p.
        \end{align*}
        Since $q \ne p^\ast$, we get the conclusion.\\
         \textbf{Case 2.} $(\alpha,\beta)=(0,1)$ and $  q \ne p_s^\ast=\frac{np}{n-sp} $.\\
       In this case, on comparing the equations \eqref{a2eq1} and \eqref{a2eq2}, we get
        \begin{align*}
            \frac{n-sp}{np}[u]_{s,p}^p = \frac1q [u]_{s,p}^p.
        \end{align*}
        Since $q \ne p_s^\ast$, we get the conclusion.\\
        \textbf{Case 3.} $(\alpha,\beta)=(1,1)$ and either $  q \geq p^\ast $ or $q \leq p_s^\ast$.\\
       In this case, on comparing the equations \eqref{a2eq1} and \eqref{a2eq2}, we get
        \begin{align*}
      \left(\frac{n-p}{np}-\frac1q\right)\| H(\nabla u)\|_{L^p(\mb R^n)}^p +   \left(   \frac{n-sp}{np}-\frac1q\right)[u]_{s,p}^p &= 0.
      \end{align*}
        Now if $q \geq p^\ast > p_s^\ast$, then $\frac{n-p}{np}-\frac1q \geq 0$ and $\frac{n-sp}{np}-\frac1q > 0$, which implies $u \equiv 0$. Similarly, if $q \leq p_s^\ast < p^\ast$, then $\frac{n-p}{np}-\frac1q < 0$ and $\frac{n-sp}{np}-\frac1q \leq 0$, which again implies $u \equiv 0$. \QED
    \end{proof}
Finally, as an application of the Pohozaev identity for the systems, we give the following nonexistence results:
\begin{Proposition}\label{sysapp}
     Let $0<s<1<p<n$ and $q>1$. Then the system of equations
\begin{align*}
 -\alpha H_p u +  \beta (-\Delta_p)^s u= \lambda |u|^{q-2}u~\text{in}~\mb R^n,\\
 -\alpha H_p v +   \beta (-\Delta_p)^s v= \mu |v|^{q-2}v~\text{in}~\mb R^n,
\end{align*}
where $\lambda,\,\mu \in \mb R$, admits no non-trivial weak solution $(u,v) \in X \times X$, where $X = X_{\alpha,\beta} \cap L^\infty(\mb R^n) \cap L^q(\mb R^n)$ and $u,v>0$ in $\mb R^n$ or $u,v<0$ in $\mb R^n$ or  $u>0,v<0$ in $\mb R^n$ or  $u<0,v>0$ in $\mb R^n$, in the following cases
\begin{enumerate}
    \item $(\alpha,\beta)=(1,0)$ and $  q \ne p^\ast=\frac{np}{n-p} $,
    \item $(\alpha,\beta)=(0,1)$ and $  q \ne p_s^\ast=\frac{np}{n-sp} $,
    \item $(\alpha,\beta)=(1,1)$ and either $q \geq p^\ast$ or $q \leq p_s^\ast$.
\end{enumerate}
\end{Proposition}
\begin{proof}
Suppose $(u,v) \in X \times X$ is a weak solution of \eqref{msys}. We prove the result by assuming both $u,v>0$ in $\mb R^n$. The proof in the other cases are similar by working with $-u$ if $u<0$ in $\mb R^n$ and by taking into account the property (H2) and Lemma \ref{Happ}.

   Let $g(t,s) = \frac{\lambda |t|^{q} + \mu |s|^{q}}{q}$. Since $u,v>0$ in $\mb R^n$, we have $g_u(u,v)$ and $g_v(u,v)$ are continuous and $g_u(u,v)u \in L^1(\mb R^n)$ and $g_v(u,v)v \in L^1(\mb R^n)$. 
   Thus, by taking $\phi =u$ in \eqref{wksysdef1} and $\psi =v$ in \eqref{wksysdef2} with  and using Lemma \ref{Happ}, we get 
     \begin{equation}\label{a2eq11}
           \alpha(\| H(\nabla u)\|_{L^p(\mb R^n)}^p + \| H(\nabla v)\|_{L^p(\mb R^n)}^p) +\beta ( [u]_{s,p}^p + [v]_{s,p}^p) = \I{\mb R^n}(ug_u(u,v) + vg_v(u,v)) \,dx.
       \end{equation}
        Also, since $u, v \in L^\infty(\mb R^n)$, we have $|g_u(u,v)| + |g_v(u,v)| \leq C (|u|^{p-1}+|v|^{p-1})$  for some $C>0$ depending on $\|u\|_{L^\infty(\mb R^n)}$ and $\|v\|_{L^\infty(\mb R^n)}$. By observing $g(u,v)= \frac{ug_u(u,v) + v g_v(u,v)}{q}$, we obtain by Theorem \ref{t1.2} that 
       \begin{equation}\label{a2eq22}
           \frac{\alpha(n-p)}{np}\left(\| H(\nabla u)\|_{L^p(\mb R^n)}^p + \| H(\nabla v)\|_{L^p(\mb R^n)}^p\right) + \frac{\beta(n-ps)}{np}\left([u]_{s,p}^p +[v]_{s,p}^p \right) = \frac1q \I{\mb R^n} (ug_u(u,v) + vg_v(u,v)) \,dx.
       \end{equation}
    Now, let us consider the three cases separately.\\
   \textbf{Case 1.} $(\alpha,\beta)=(1,0)$ and $  q \ne p^\ast=\frac{np}{n-p} $.\\
       In this case, on comparing the equations \eqref{a2eq11} and \eqref{a2eq22}, we get
        \begin{align*}
            \frac{n-p}{np}\left(\| H(\nabla u)\|_{L^p(\mb R^n)}^p + \| H(\nabla v)\|_{L^p(\mb R^n)}^p \right)   = \frac1q \left(\| H(\nabla u)\|_{L^p(\mb R^n)}^p +\| H(\nabla u)\|_{L^p(\mb R^n)}^p \right).
        \end{align*}
        Since $q \ne p^\ast$, we get the conclusion.\\
         \textbf{Case 2.} $(\alpha,\beta)=(0,1)$ and $  q \ne p_s^\ast=\frac{np}{n-sp} $.\\
       In this case, on comparing the equations \eqref{a2eq11} and \eqref{a2eq22}, we get
        \begin{align*}
            \frac{n-sp}{np}([u]_{s,p}^p + [v]_{s,p}^p) = \frac1q ([u]_{s,p}^p + [v]_{s,p}^p).
        \end{align*}
        Since $q \ne p_s^\ast$, we get the conclusion.\\
        \textbf{Case 3.} $(\alpha,\beta)=(1,1)$ and either $  q \geq p^\ast $ or $q \leq p_s^\ast$.\\
       In this case, on comparing the equations \eqref{a2eq11} and \eqref{a2eq22}, we get
        \begin{align*}
      \left(\frac{n-p}{np}-\frac1q\right)(\| H(\nabla u)\|_{L^p(\mb R^n)}^p +H(\nabla v)\|_{L^p(\mb R^n)}^p ) +   \left(   \frac{n-sp}{np}-\frac1q\right)([u]_{s,p}^p +[v]_{s,p}^p) &= 0.
      \end{align*}
        Now if $q \geq p^\ast > p_s^\ast$, then $\frac{n-p}{np}-\frac1q \geq 0$ and $\frac{n-sp}{np}-\frac1q > 0$, which implies $u,\,v \equiv 0$. Similarly, if $q \leq p_s^\ast < p^\ast$, then $\frac{n-p}{np}-\frac1q < 0$ and $\frac{n-sp}{np}-\frac1q \leq 0$, which again implies $u,\,v \equiv 0$. \QED
\end{proof}

 \begin{Remark}
        In particular, if $p =q$ in Proposition \ref{PNE}, then $f$ satisfies $(f_1)$ for any weak solution $u\in X_{\alpha,\beta}$ of \eqref{meqn}. Therefore, under the hypotheis of Proposition \ref{PNE}, we conclude that the eigenvalue problem 
         \begin{align*}
        -\alpha H_p u + \beta (-\Delta_p)^s u = \lambda u_+^{p-1} - \mu u_-^{p-1}~\text{in}~\mb R^n
    \end{align*}
    where $\lambda ,\mu \in \mb R$, has no non-trivial weak solution in $X_{\alpha,\beta}$. Moreover,  if $p =q$ in Proposition \ref{sysapp}, then $g$ satisfies $(g_1)$ for any weak solution $(u,v)\in X_{\alpha,\beta}\times X_{\alpha,\beta}$.  Then under the hypotheis of Proposition \ref{sysapp},  the system of equations
\begin{align*}
 -\alpha H_p u +  \beta (-\Delta_p)^s u= \lambda |u|^{p-2}u~\text{in}~\mb R^n,\\
 -\alpha H_p v +   \beta (-\Delta_p)^s v= \mu |v|^{p-2}v~\text{in}~\mb R^n,
\end{align*}
where $\lambda,\,\mu \in \mb R$, admits no non-trivial weak solution $(u,v) \in X_{\alpha,\beta}  \times X_{\alpha,\beta} $ where $u,v>0$ in $\mb R^n$ or $u,v<0$ in $\mb R^n$ or  $u>0,v<0$ in $\mb R^n$ or  $u<0,v>0$ in $\mb R^n$.
\end{Remark}

 \section*{Acknowledgement} P. Garain thanks IISER Berhampur for the Seed grant: IISERBPR/RD/OO/2024/15, Date: February 08, 2024. A part of this work was completed when the first author was a PDRF at IISER Berhampur. The first author wants to thank IISER Berhampur for its support and hospitality.

\end{document}